  \documentclass[12pt,oneside,reqno]{amsart}

\usepackage{amsmath}
   \usepackage{amssymb}
   \usepackage{graphicx}

      \theoremstyle{plain}
      \newtheorem{thm}{Theorem}[section]
      \newtheorem{proposition}[thm]{Proposition}
      \newtheorem{lemma}[thm]{Lemma}
      \newtheorem{cor}[thm]{Corollary}
      \theoremstyle{definition}
      \newtheorem{defn}{Definition}
       \newtheorem{remark}{Remark}[section]    
      
      \numberwithin{equation}{section}

      \newcommand{\nn}{\nonumber}
      \newcommand{\R}{\mathbb{R}}

      \newcommand{\Z}{\mathbb{Z}}
      \newcommand{\Vsp}{\phantom{\Big\vert}}

\begin{document}

\title[Stick Index of Knots and Links in the Cubic Lattice]{Stick Index of Knots and Links in the Cubic Lattice}

\author[Adams]{Colin Adams}
\address{Colin Adams, Department of Mathematics and Statistics, Williams College, Williamstown, MA 01267}
\email{Colin.C.Adams@williams.edu}
\author[Chu]{Michelle Chu}
\address{Michelle Chu, Department of Mathematics, University of Texas, 1 University Station, C1200
Austin, Texas 78712}
\email{mchu@math.utexas.edu}
\author[Crawford]{Thomas Crawford}
\address{Thomas Crawford, Department of Mathematics and Statistics, Williams College, Williamstown, MA 01267}
\email{Thomas.N.Crawford@williams.edu}
\author[Jensen]{Stephanie Jensen}
\address{Stephanie Jensen, Department of Mathematics and Statistics, Williams College, Williamstown, MA 01267}
\email{Stephanie.A.Jensen@williams.edu}
\author[Siegel]{Kyler Siegel}
\address{Kyler Siegel, Department of Mathematics, Stanford University, Department of Mathematics, Building 380, Stanford, California 943057}
\email{kylers@math.stanford.edu}
\author[Zhang]{Liyang Zhang}
\address{Liyang Zhang, Department of Mathematics and Statistics, Williams College, Williamstown, MA 01267}
\email{Liyang.Zhang@williams.edu}

\keywords{lattice stick number, torus knots, satellite knots, composite knots}
\subjclass[2000]{57M50}

\begin{abstract}
The cubic lattice stick index of a knot type is the least number of sticks necessary to construct the knot type in the $3$-dimensional cubic lattice. 
We present the cubic lattice stick index of various knots and links, including all $(p,p+1)$-torus knots, and show how composing and taking satellites can be used to obtain the cubic lattice stick index for a relatively large infinite class of knots. Additionally, we present several bounds relating cubic lattice stick index to other known invariants. 
\end{abstract} 

\date{\today}
\maketitle

\begin{section}{Introduction}
An important invariant in physical knot theory is the stick index, namely the least number of sticks (i.e. line segments) needed to construct a given knot in $3$-space. In this paper we are interested in a related invariant, the cubic lattice stick index, which is the least number of sticks (without regard to length) necessary to construct a knot in the $3$-dimensional cubic lattice $(\R \times \Z \times \Z) \cup (\Z \times \R \times \Z) \cup (\Z \times \Z \times \R)$. Since we do not refer to other lattices, we will usually omit the word {\em cubic}. Throughout, we will let $K$ represent a fixed conformation of a knot and $[K]$ the collection of all conformations of a given knot type.We begin with some definitions:

\begin{defn}
A {\em stick} of a lattice stick conformation $K$ is a maximal line segment of $K$ in $\R^3$, i.e. a line segment contained in $K$ that is not contained in any longer line segment of $K$. An {\em x-stick} of $K$ is a stick parallel to the $x$-axis. We define $y$-sticks and $z$-sticks similarly.
\end{defn}

\begin{defn}
The {\em lattice stick number} $s_{CL}(K)$ of a lattice stick conformation $K$ is the number of constituent sticks of $K$. We 
minimize $s_{CL}(K)$ over all lattice stick conformations $K \in [K]$ and we call this number the {\em lattice stick index} $s_{CL}[K]$ of the knot type $[K]$. A lattice stick conformation realizing this minimum is called a {\em minimal lattice conformation}. 
\end{defn}

Previously,  lattice stick index was known for three knots. In 1993, Diao proved that $s_{CL}[3_1]=12$ (\cite{Diao}).
Then in 1999, Promislow and Rensburg proved that $s_{CL}[9_{47}]=18$ (\cite{Promislow and Rensburg}), and in 2005 Huh and Oh proved that $s_{CL}[4_1]=14$ (\cite{Huh and Oh}).
Huh and Oh also proved in 2010 that the trefoil and figure eight are the only knots with 
lattice stick index less than $15$ (\cite{Huh and Oh 2}).

In Section \ref{determining} we present a lower bound on lattice stick index based on bridge number. We use this to find the lattice stick indices of $8_{20}$, $8_{21}$, and $9_{46}$, as well as all $(p,p+1)$-torus knots.
In Section \ref{composition}, we show how {\em exterior L}'s and {\em clean L}'s can be used to compose minimal lattice stick conformations to obtain the minimal lattice stick conformation of the composite knot.
In Section \ref{satellites}, we show how known minimal lattice stick conformations can be used to determine lattice stick indices for  various satellite knots. 
In Section \ref{crossing}, we present both upper and lower bounds on $s_{CL}$ in terms of crossing number.
Finally, in Section \ref{links}, we consider the minimal lattice stick index of certain links of more than one component.

\label{intro}
\end{section}

\begin{section}{Determining Cubic Lattice Stick Index}

We use a lower bound based on bridge index and an upper bound based on construction to determine the lattice stick index of many knots.

 \begin{defn} The {\bf bridge index} of a knot is given by      
  $$b[K] = \underset{K \in [K]}{\text{min}}\: \underset{\mathbf {v} \in S^2}{\text{min}} (\text{\# of local maxima when K is projected to }  \mathbf{v})$$ 
      \end{defn}
      
      Bridge index was first introduced in \cite{Schubert} and has proved to be an extremely useful invariant.  The following lemma was also found by Promislow and Rensburg in \cite{Promislow and Rensburg}.

\begin{lemma}
\label{lem:lowerbound} $s_{CL} [K] \geq 6b[K]$.
\end{lemma}

\begin{proof} 
Let $K \in [K]$ be a minimal lattice stick conformation. Then $K$ must have at least $b[K]$ local maxima in each of the $x,y$ and $z$ directions.
First considering the $z$ direction, let $S$ be a portion of $K$ corresponding to a local $z$-maximum
and contained in a certain $z$-level. Let $s_1$ and $s_2$ be the first $z$-sticks we encounter as we continue along $K$ in either direction, starting at $S$.
Since $S$ corresponds to a local maximum, $s_1$ and $s_2$ must both connect the $z$-level of $S$ to lower $z$-levels. 
Then for each local $z$-maximum of $K$ we have two distinct corresponding $z$-sticks, and it follows that $K$ must contain at least $2b[K]$ $z$-sticks.
A similar argument for the $x$ and $y$ directions shows that $K$ must contain at least $2b[K]$ sticks in each of the three primary directions.
\end{proof}
Note that by the same argument, we must have two $z$-sticks corresponding to each local $z$-minimum, and similarly for the $y$ and $z$
directions. Then if $K$ is a lattice stick conformation constructed with exactly $6b[K]$ sticks, each stick must connect a local minimum to a local maximum
in its own direction and lie at a local extremum in the other two primary directions.

\begin{thm}
Knots $8_{20}$, $8_{21}$, and $9_{46}$ all have lattice stick index $18$.
\end{thm}

\begin{proof}
Knots $8_{20}$, $8_{21}$, and $9_{46}$ are all $3$-bridge knots. Lemma \ref{lem:lowerbound} and Figure \ref{sporadicknots} complete the proof.

\end{proof}

\begin{figure}[h]
\center
\includegraphics[height=1.3in]{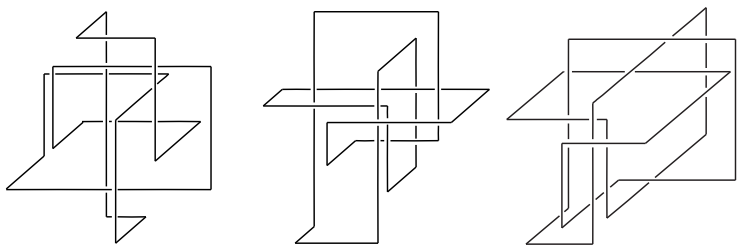}
\caption{Minimal lattice conformations of $8_{20}$, $8_{21}$, and $9_{46}$.}
\label{sporadicknots}
\end{figure}

Next we find the lattice stick index for an infinite class of prime knots, namely the $(p,p+1)$-torus knots, which we denote $T_{p, p+1}$. 

\begin{thm} $s_{CL}[T_{p, p+1}] = 6p$ for $p\geq 2$. \end{thm}
\begin{proof} By \cite{Schubert}, we know $b[T_{p, p+1}]=p$. Figure \ref{torus} shows a construction of the first three knots
 in the class of $(p,p+1)$-torus knots using $6p$ sticks. Using the pattern in Figure \ref{torus}, we can construct the rest of 
the $(p, p+1)$-torus knots with $6p$ sticks. Lemma \ref{lem:lowerbound} completes the proof.
\begin{figure}
\includegraphics[height=1.5in]{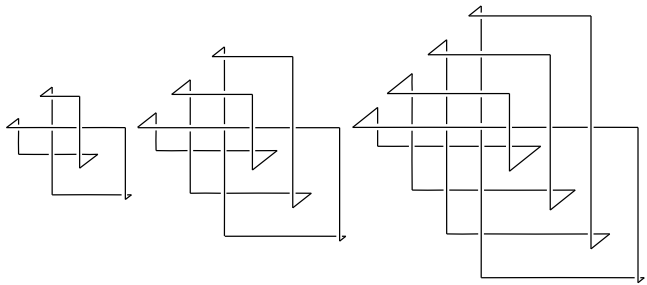}
\caption{Minimal lattice conformations of $T_{2, 3}, T_{3, 4}$, and $T_{4, 5}$.}
\label{torus}
\end{figure}
\end{proof}

\label{determining}
\end{section}

\begin{section}{Composition}

We begin by defining clean $L$'s and exterior $L$'s of a lattice stick conformation.

\begin{defn}
Let $S$ be a portion of a lattice stick conformation $K$ consisting of four consecutive sticks $s_1,s_2,s_3,s_4$, 
where $s_1$ and $s_4$ are parallel to each other and $s_2$ and $s_3$ are perpendicular  both to each other and to $s_1$ and $s_4$.
Let $R$ be the closed rectangle with two edges of which are given by $s_2$ and $s_3$. If $s_1$ and $s_4$ lie to the same side of the plane containing $R$ and 
$R \cap K \subset S$, we call $S$ a {\em clean L} of $K$.
If in addition $K$ is entirely contained in one of the two closed half-spaces defined by the plane containing $R$, we call
$S$ an {\em exterior L}.
\end{defn}

In \cite{Schubert}, Schubert showed that $b[K_{1}\#K_{2}]=b[K_{1}]+b[K_{2}]-1$. 
We use this fact and Lemma \ref{lem:lowerbound} to prove the following theorem for lattice stick index under composition.

\begin{thm}
Let $K_{1}$ and $K_{2}$ be minimal lattice conformations realizing $s_{CL}[K_{i}]=6b[K_{i}]$ for $i=1,2$ such that $K_{1}$ 
has an exterior $L$ and $K_{2}$ has a clean $L$. Then $s_{CL}[K_{1}\#K_{2}]=s_{CL}[K_{1}]+s_{CL}[K_{2}]-6$.
\end{thm}

\begin{proof} 
Let $L_1$ be an exterior $L$ of $K_1$ and let $L_2$ be a clean $L$ of $K_2$. Let $R_1$ and $R_2$ be the closed rectangles defined by the middle two sticks of $L_1$ and $L_2$ respectively. Assume that $R_2$ lies at the level $z=0$ with the first and last sticks of $L_2$ lying in the region $z \leq 0$. 
We now $(+z)$-translate all vertices of $K_2$ lying in $z > 0$ by some large integer $n_0$, and modify the sticks of $K_2$ accordingly so that $K_2$ is still a minimal lattice stick conformation of the same knot type, but with no $x$-sticks or $y$-sticks in the region $0 < z \leq n_0$.
We call this move {\em expanding $K_2$ in the half-space $z\geq0$ by $n_0$}.

Now let $P_1,P_2,P_3,P_4$ denote the four planes that are perpendicular to the plane containing $R_2$ and that contain an edge of $R_2$. 
Let $S_i$ denote the closed half-space defined by $P_i$ that does not contain $R_2$.
Similar to before, for $n=1,2,3,4$ we expand the $K_2$ in $S_i$ by $n_i$, where $n_i$ is a sufficiently large integer.

The resulting $K_2$ after these five expansions is still a minimal lattice stick conformation of the same knot type. The advantage is that we are now in a position to place $K_1$ (i.e. rotate and translate it) such that it lies in $z\geq 0$ with $R_1$ lying at $z=0$, and such that $K_1$ and $K_2$ are disjoint except possibly
at $L_1$ and $L_2$. By scaling up $K_2$ in either the $x$ or $y$ direction, we can make $R_2$ a square. We can then 
scale up $K_1$ (without it intersecting $K_2$, since $n_0,...,n_4$ were chosen sufficiently large) such that $R_1$ is a square of the same size.
After a possible rotation, $R_1 \cap K_1$ and $R_2 \cap K_2$ are now identical. We remove $R_1$ and $R_2$ and consolidate the remaining four sticks of $L_1$ and $L_2$ into two $z$-sticks (see Figure \ref{Flo:L3}). The result is a lattice stick conformation of $[K_1\#K_2]$ with $s_{CL}[K_1] + s_{CL}[K_2] - 6$ sticks, which is minimal by Lemma \ref{lem:lowerbound}. 

\begin{figure}[b]
\includegraphics[height=1.3in]{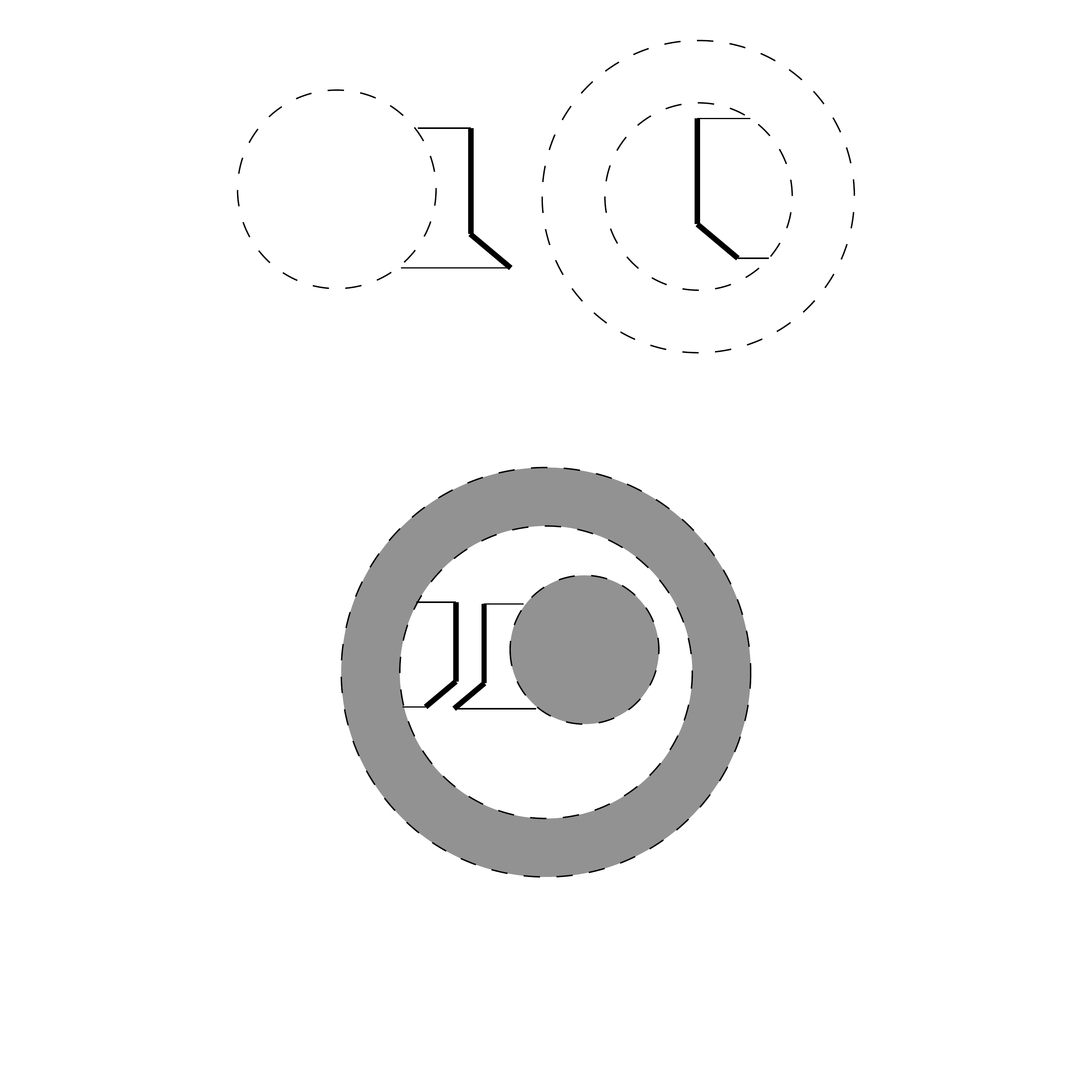}
\includegraphics[height=1.3in]{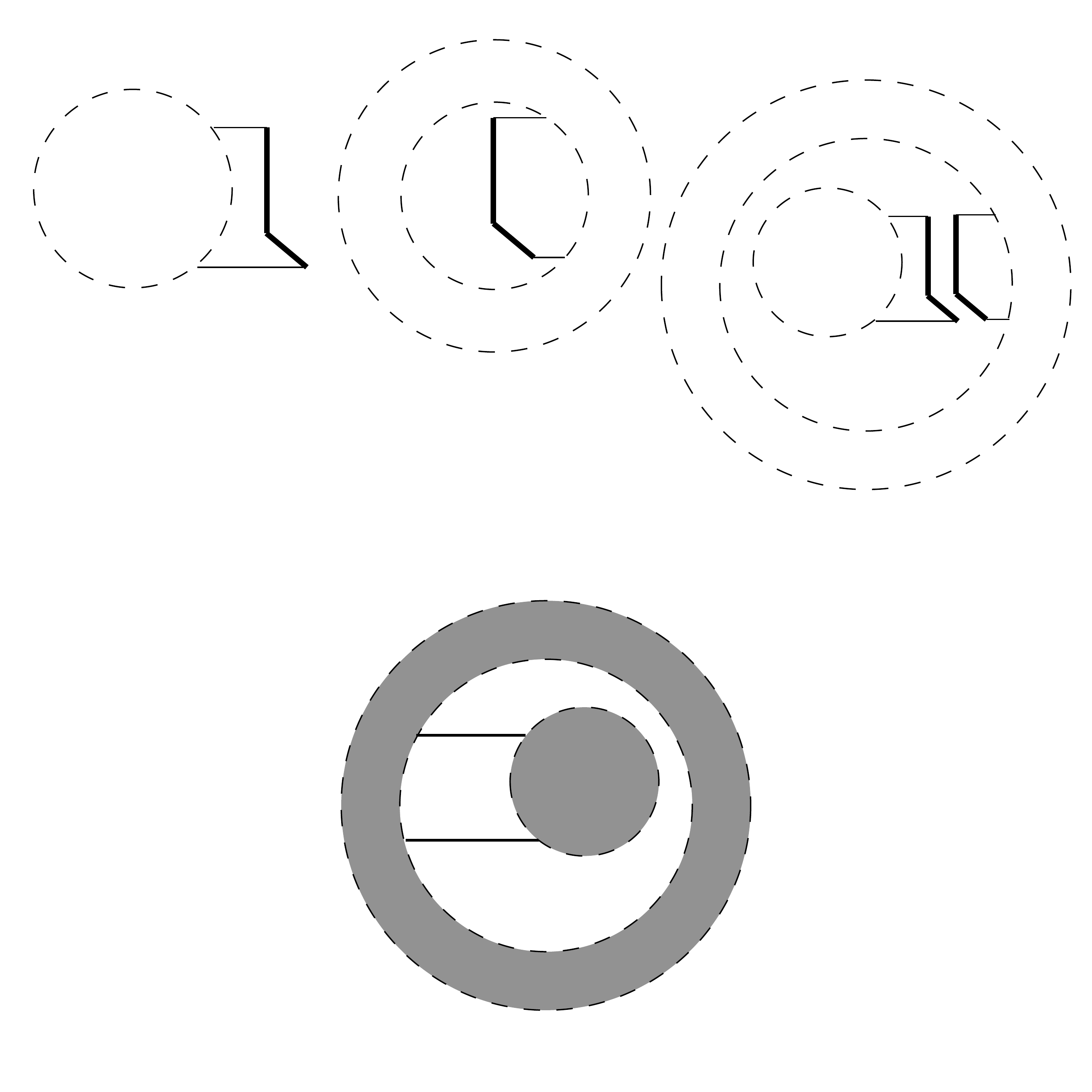}\caption{Gluing $L_1$ and $L_2$ by eliminating $6$ sticks.}
\label{Flo:L3}
\end{figure}
\end{proof}

The following corollary follows by induction.

\begin{cor}
For $K_{1}, \dots, K_{n}$ minimal lattice knot conformations realizing $s_{CL}[K_{i}]=6b[K_i]$
for $i=1,...,n$ each possessing one exterior L and at least one other clean
L, $s_{CL}[K_{1}\# \dots \#K_{n}]=s_{CL}[K_{1}]+ \dots +s_{CL}[K_{n}]-6(n-1)$. 
\end{cor}

\label{composition}
\end{section}

\begin{section}{Lattice Stick index of Certain Satellite Knots}

In this section we find the cubic lattice stick index of certain $n$-string  satellite knots.  We show that if we are given a companion knot $[J]$ with $s_{CL}[J]=6b[J]$, certain satellites of $[J]$ have cubic lattice stick index $6nb[J]$.

A \emph{braid pattern} $B$ is defined to be a conformation of any knot in braid form that can be contained in a standardly embedded solid torus such that any meridinal disk of the torus intersects $B$ at least $n$ times, where $n$ is the number of strands.  To form a satellite knot, the solid torus containing the pattern is homeomorphically mapped to a neighborhood of the companion knot, carrying the pattern along.   We say that a \emph{braid satellite knot} is a satellite knot that has a braid pattern.  We only consider braid patterns with a limited number of crossings, depending on the conformation of the companion knot.

\begin{defn}
We say a stick $S$ in the cubic lattice has torsion $\frac{\pi}{2}$ if when we project to a perpendicular plane (i.e. ``look down $S$"), the projection of the adjacent sticks form an angle of $\frac{\pi}{2}$. We call $S$ a \emph{torsion stick}. 
\end{defn}

Note that every lattice stick conformation of a nontrivial knot cannot lie in a plane and therefore has at least two torsion sticks.
Let $J$ be a minimal lattice conformation of $[J]$. Let $J_n$ be the conformation of the $n$ component link formed by duplicating $J$ $(n-1)$ times and translating each copy by the vector $(1,1,-1)$ from the previous.  We assume that we have scaled $J$ up large enough so that no two copies intersect and so that all copies remain isotopic to one another in the resulting link complement.

\begin{defn}
Define $\omega$, a braid word, to be a \emph{permutation word} if any two strings of $w$ cross each other at most once and each crossing is in the same direction (all right over left or all left over right). The second condition is equivalent to the braid being either a positive braid or  a negative braid.
\end{defn}

\begin{defn} We say a conformation of knot or link  $K$ \emph{differs from $J_n$ by  a permutation word $\omega$} if there exist $n$ parallel sticks of $J_n$ obtained from the translation process applied to a single stick of J which together yield a trivial braid, that when replaced by the permutation word $\omega$ yields $K$.
\end{defn}

\begin{proposition}
Satellite knots $K$ that differ from $J_n$ by a single permutation word $\omega$ have cubic lattice stick index 6nb[J].
\end{proposition}

\begin{proof}
 By \cite{Schubert},  $b[K] \geq nb[J]$.  Thus, $s_{CL}[K]\geq 6nb[J]$.   Note that $J_n$ has $6nb[J]$ sticks.  We claim there is a transformation from $J_n$ to $K$ which preserves the number of sticks.

Note that the particular stick in $J$ to which we apply the process of replacing the trivial braid given by the $n$ corresponding parallel sticks in $J_n$ by the braid given by $\omega$ is irrelevant, as all of the resulting knots will be equivalent. Since $J$ is a knotted polygon in space and hence does not lie in a plane, there exists a torsion stick, say $a$. Label the two adjacent sticks $b$ and $c$.  Label the sticks in $J_n$ corresponding to $b$ and $c$ as $b_1, b_2, \dots , b_n$ and $c_1, c_2, \dots ,c_n$ respectively.  Projecting down $a$, we see a picture as in Figure \ref{Flo:S1}.
The permutation word $\omega$ requires each $b_i$ to connect to some $c_j$. To do this, extend or shorten sticks $b_i$ and $c_j$ by $j-i$ units and connect them with a vertical stick. (See Figure \ref{Flo:S2}.)  This construction of $K$ with $6nb[J]$ sticks gives us $s_{CL}[K] \leq 6nb[J]$. 
\end{proof}

\begin{figure}
\includegraphics[height=1.3in]{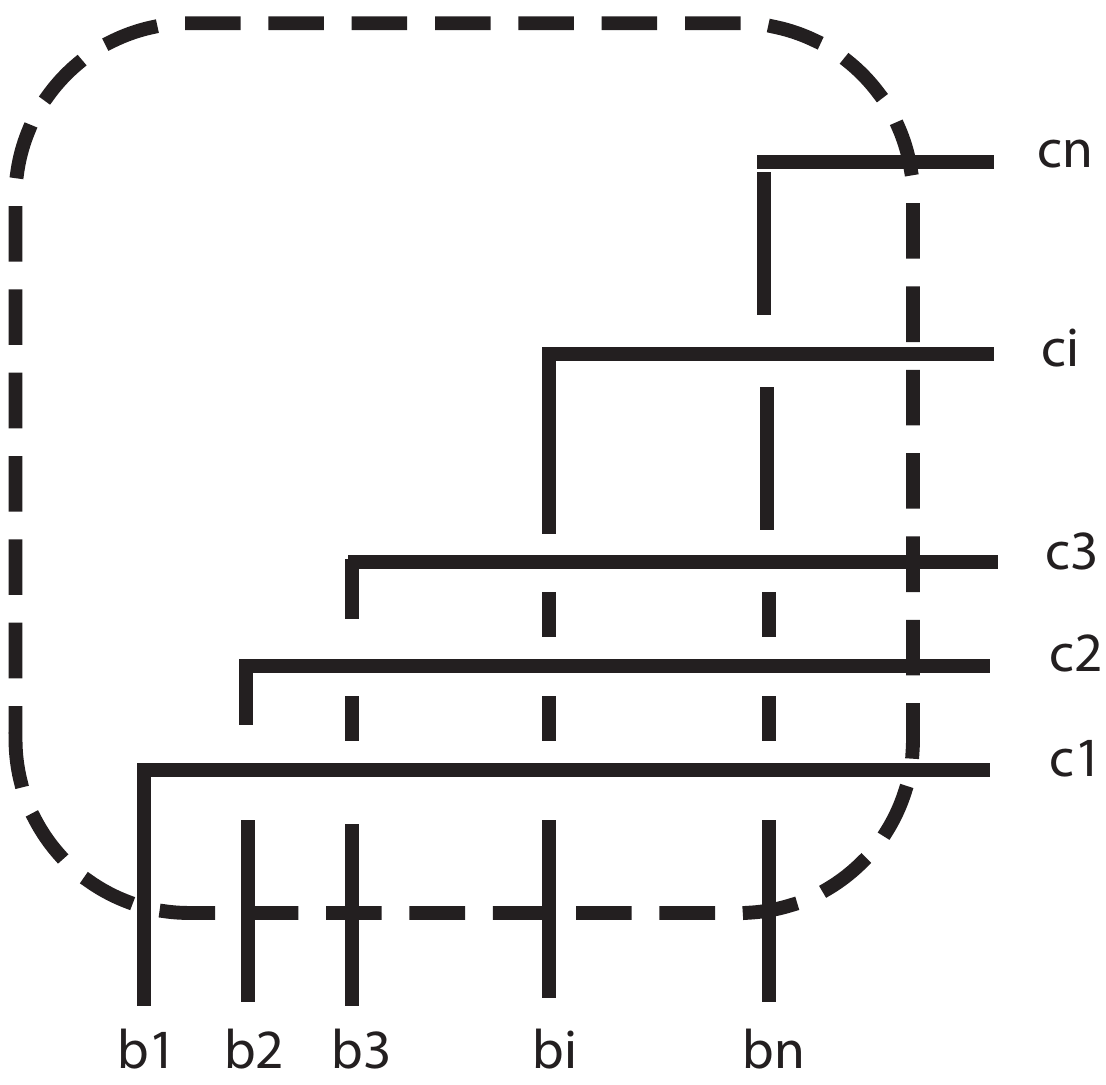}
\label{Flo:S1}
\caption{Projection to the plane perpendicular to $a$.}
\end{figure}

\begin{figure}
\includegraphics[height=1.3in]{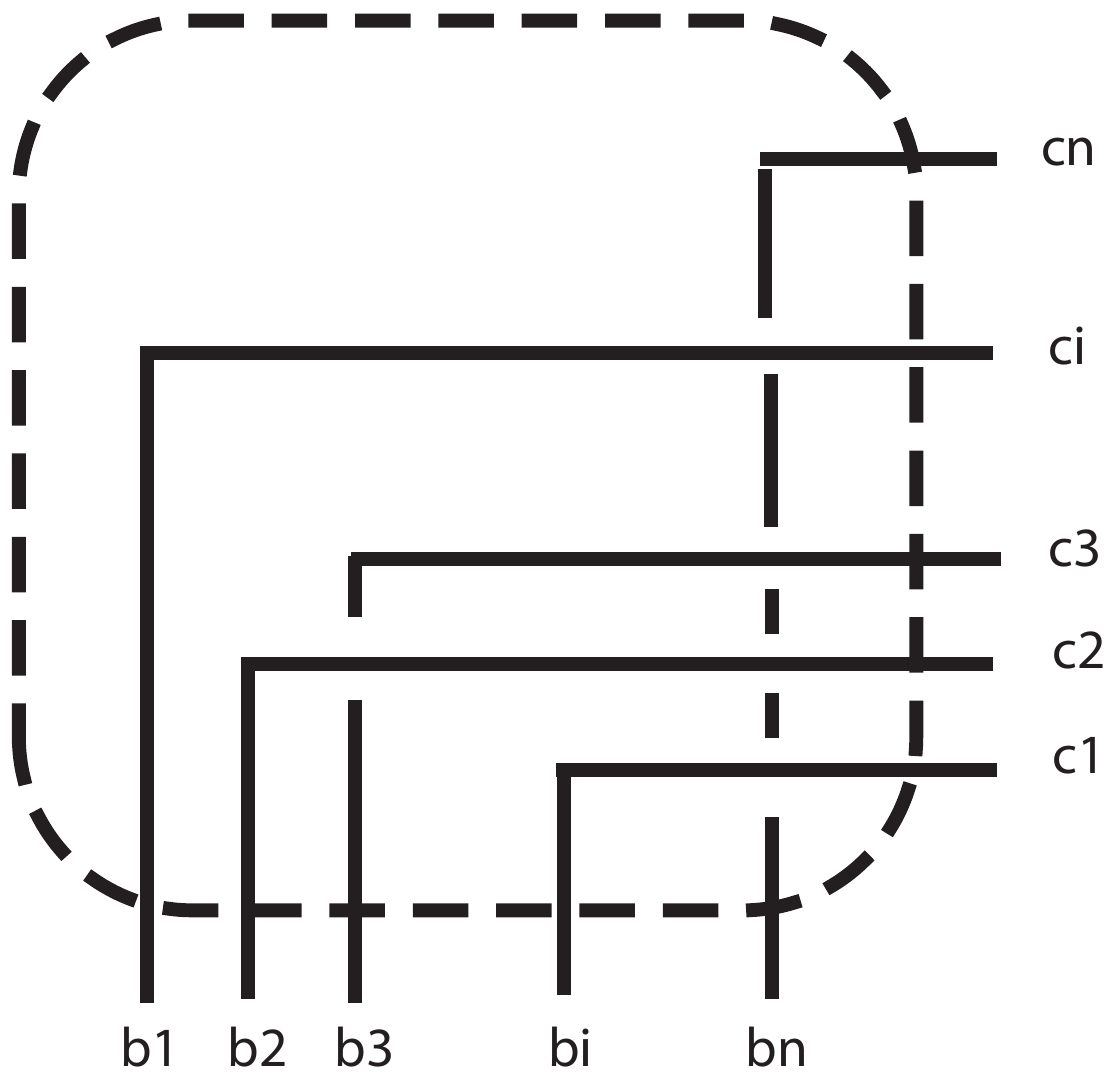}
\caption{Permutation given $\omega(b_1)=c_i$ and $\omega(b_i)=c_1$.}
\label{Flo:S2}
\end{figure}

\begin{cor}
Let $S_n$ be the collection of knots which differ from $J_n$ by at most one permutation word for each torsion stick of J.  All knots in $S_n$ have cubic lattice stick index $6nb[J]$.  
\end{cor}

\begin{proof}
Since the algorithm in the previous proposition only changes the positions of one set of $n$ sticks (the set corresponding to the original torsion stick), each torsion stick can allow $K$ to differ from $J_n$ by another permutation word.  So all knots $K\in S_n$ can be constructed in $6nb[J]$ sticks.  Therefore $s_{CL}[K]=6nb[J]$ for $K\in S_n$.
\end{proof}

\begin{figure}
\includegraphics[height=2.8in]{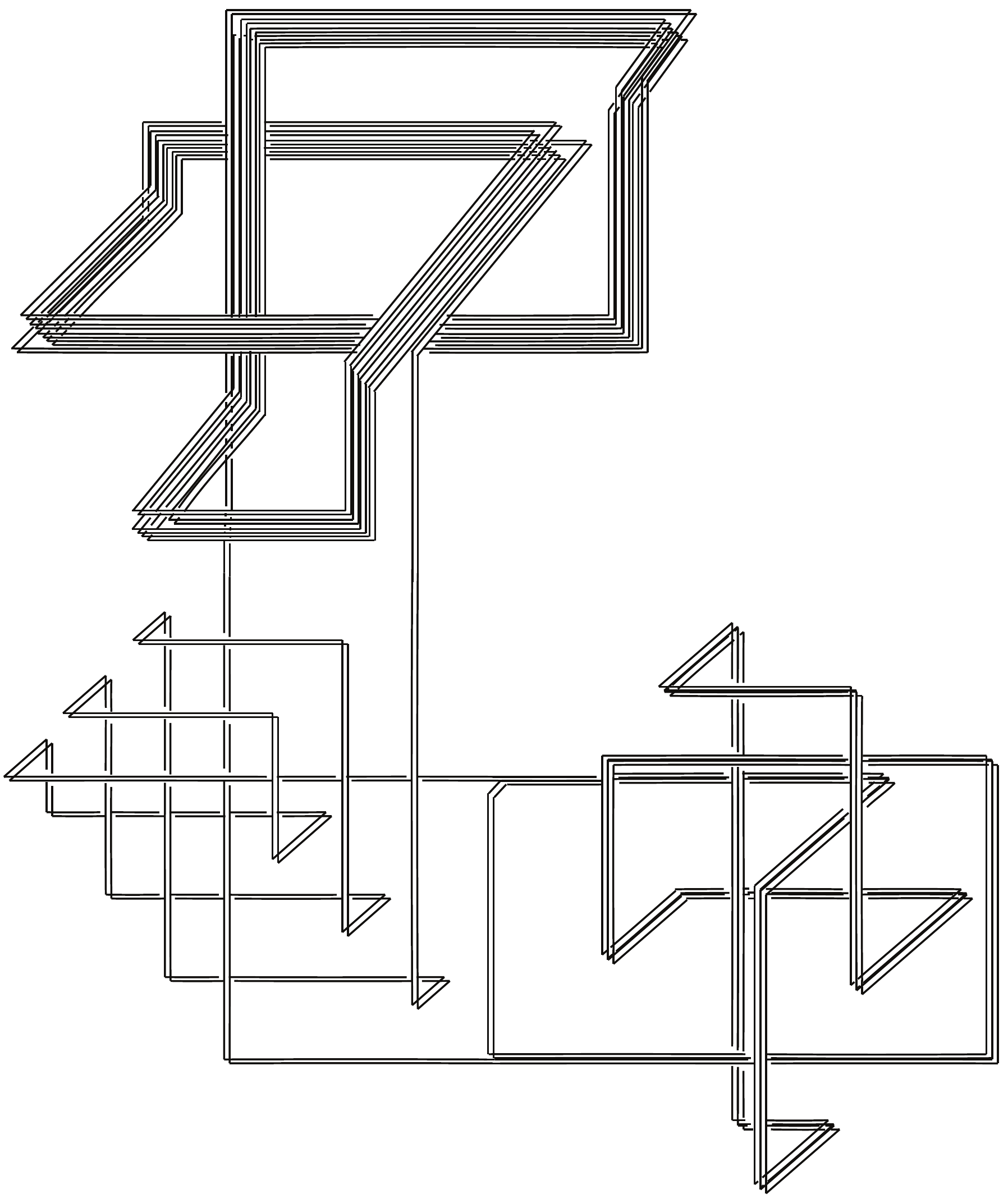}
\caption{A $2$-string satellite of (a $5$-string satellite of $3_1$ composed with the $(4,5)$-torus knot composed with a $2$-string satellite of $9_{46}$) with cubic lattice stick index $216$.}
\label{biscuit}
\end{figure}

We observe that taking satellites and composing knots are two operations that can be done independently, allowing us to determine cubic lattice stick index for a variety of different knots and knot classes. (See Figure \ref{biscuit}).

\label{satellites}
\end{section}

\begin{section}{Crossing Number Bounds}
For a lattice stick conformation, let $|P_x|$, $|P_y|$, and $|P_z|$ denote the number of $x,y$ and $z$ sticks respectively.
Following Huh and Oh in \cite{Huh and Oh}, we observe that we can always modify the lengths of sticks of a lattice 
knot conformation $K$, without changing the number of sticks, and translate the knot so that the vertices of $K$ lie in 
the set $\{1,...,|P_x|\} \times \{1,...,|P_y|\} \times \{1,...,|P_z|\}$, and such that exactly two $x$-sticks have an endpoint at $x = i$ for any $i \in \{1,...,|P_x|\}$. We can do the same for the $y$-sticks and $z$-sticks concurrently. We call such a lattice knot conformation {\em properly leveled}. We can also extend this concept to links, the only distinction occurring when a trivial component lies entirely in a plane, say $x = k$. In this case we can arrange it so that {\em no} $x$-sticks have endpoints at the level $x = k$, but we will need more than $|P_x|$ levels. Suppose a lattice link conformation $K$ has exactly $C_x$ trivial components lying in planes perpendular to the $x$-axis. Then we will call $K$ {\em properly leveled} if the endpoints of $x$-sticks lie in a subset of the levels $x=1\;,...,\;x=|P_x| + C_x$, where each level contains either two endpoints of $x$-sticks or a component of $K$, with similar conditions for $y$ and $z$. As with lattice knots, we observe that every lattice link can be made properly leveled without altering the number of sticks in the conformation. 

\begin{remark}\label{CanMakeKnotDiagram}
Let $K$ be a properly leveled cubic lattice conformation of a link (in particular it could be a knot), and let $P$ be the 
projection of $K$ down the $z$-axis. Note that $P$ is not necessarily a valid link diagram (i.e. a planar projection of a link conformation in $\R^3$ such that no more than two points project to the same point, where double points are transverse intersections).  Suppose that several $x$-sticks $s_1,\dots ,s_n$ all project to the same line $y = k$ in $P$. Unless there is a trivial component at $y = k$, there are exactly two $y$-sticks, $t_a$ and $t_b$, which are connected to the level $y = k$. Let us slightly $y$ translate $s_1,\dots,s_n$ such that each lies on a different line, and such that the $s_i$'s are in the $y$
order given by the $z$ coordinates (see Figure \ref{bwlatticeperturbedandunperturbed}). We can add very short line segments corresponding to the $z$-sticks and shorten or lengthen $t_a$ and $t_a$, and continue this procedure for every set of line segments whose images under $z$-projection lie on the same line (using an analogous procedure to separate $y$-sticks) such that the result is a link diagram $P'$. Here $P'$ is a $z$-projection of a slight perturbation of $K$ (not necessarily remaining in the cubic lattice), made up of line segments corresponding to the sticks of $K$, where $z$-stick images are very short and are not involved in any crossings.

\begin{figure}
 \centering
 \includegraphics[clip=true,scale=.35]{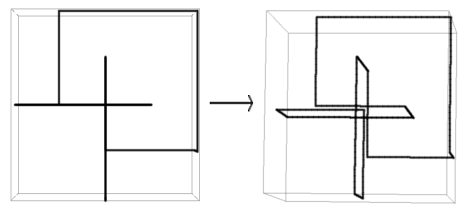}
 \caption{Perturbing a primary projection.} \label{bwlatticeperturbedandunperturbed}
\end{figure}

The advantage of $P'$ is that as a link diagram we can use it to make statements about crossing number, linking number, and 
other invariants defined via link diagrams. We switch between $P$ and $P'$ as is convenient.
\end{remark}
  
We now give bounds relating the cubic lattice stick index $s_{CL}[K]$ of a knot $[K]$ to its crossing index $c[K]$. 
We abbreviate $s_{CL}[K]$ and $c[K]$ by $s_{CL}$ and $c$ respectively. 

\begin{thm}
$s_{CL}[K] \geq 3\sqrt{c[K] + 2}$. 
\end{thm}

\begin{proof} Let $P$ and $P'$ be projections of $K$ down the $z$-axis, as in Remark \ref{CanMakeKnotDiagram}. Since every crossing in $P'$ arises from an intersection between an $x$-stick image and a $y$-stick image, we have immediately \begin{align}
c \leq (|P_x|)(|P_y|). \nn
\end{align}

After a possible renaming we can assume $|P_x| \leq |P_z|$ and $|P_y| \leq |P_z|$. 
Then $|P_x| + |P_y| \leq \dfrac{2s_{CL}}{3}$, so $(|P_x|)(|P_y|)$ is certainly bounded from above by the maximization of $(|P_x|)(|P_y|)$ constrained by
\begin{align}
|P_x| + |P_y| = \dfrac{2s_{CL}}{3}.\nn
\end{align}
This gives us
\begin{align*}
c \leq \dfrac{s_{CL}^2}{9}.
\end{align*}
We can  improve this bound slightly by noting that whenever an $x$-stick and a $y$-stick are directly attached or attached via a $z$-stick in between, there is no crossing in $P'$ between these two sticks. This must occur at least twice in $P'$, so we have
\begin{align}
c \leq \dfrac{s_{CL}^2}{9} - 2.
\end{align}

Inverting this inequality yields the result.
\end{proof}

We now find an upper bound for $s_{CL}$ as a function of $c$. Recall the following definitions as can be found in \cite{Cromwell}.

\begin{defn}
An {\em arc presentation} of a knot is an open-book presentation in which there are finitely many half-plane pages,  such that each page meets the knot in a single simple arc.
The {\em arc index} of a knot $[K]$, denoted $\alpha[K]$, or simply $\alpha$,
 is  the least number of pages in any arc presentation of $[K]$
  \end{defn}

\begin{lemma}
$s_{CL}[K] \leq 6\alpha[K] - 16$.
\end{lemma}

\begin{proof}
Let $K$ be an minimal arc presentation with $\alpha$ pages having the $z$-axis as the binding axis. Ordering the pages based on their angle $\theta$ from the positive $x$-axis, isotope $K$ such that the first three pages lie at $\theta = \pi/2$, $\theta = \pi$, and  $\theta = 3\pi/2$ respectively, and such that the remaining pages lie in $3\pi/2 < \theta < 5\pi/2$. The arc in the second page cannot connect to either of the arcs in the first and third page, as if it did, we could lower the arc index. Let $P_a$ and $P_b$ be the pages of the arcs that this arc does connect to.

       Isotope the arcs of $K$ lying in each of the first three pages without leaving the pages such that they each consist of three sticks, one parallel to the $z$-axis and two parallel to either the $x$-axis or $y$-axis (see Figure \ref{Lattice Arc Index Diagram}). Isotope page $P_a$ to be the page given by $\theta = 2\pi$. The arc in this page can be realized by extending the connecting stick from the second page a distance     $\alpha$ units and adding two other sticks in that half-plane, one $z$-stick, and one $x$-stick taking us back to the $z$-axis. The arc in plane $P_b$ can be realized by extending another stick from the second plane across the $z$-axis into the half-plane given by $\theta= 2\pi$, and then adding four more sticks, one $y$-stick to leave the half-plane corresponding to $\theta = 2\pi$, one $z$-stick, one $y$-stick to return to the half-plane $\theta=2\pi$ and one $x$-stick to return to the $z$-axis. The $y$-sticks will be to the negative $y$ or positive $y$ side of the $\theta= 2 \pi$ plane depending on whether $P_b$ comes before or after $P_a$ in the sequence of half-planes. We can choose the length of the $x$-stick shorter than the $x$-stick in $P_a$ to allow this arc to pass over or under the arc in $P_a$ as needed.  
       
       Now all remaining arcs will  start in the half-plane $\theta = 2\pi$, leaving that plane below or above it, depending on whether their original plane came before or after $P_a$ in the sequence of pages. Orient the entire knot. For the two arcs that follow the arcs in the first and third page, we simply take an $x$-stick followed by a $y$-stick, a $z$-stick, a $y$-stick and an $x$-stick. For all other planes, we require a short $z$-stick together with the other sticks mentioned above. Finally, to close up at the end, we need another $z$-stick. We can pick the small $z$-sticks lying on the $z$-axis appropriately such that the result is a valid lattice conformation which is ambient isotopic to the original $K$. 
       
Thus, if there are at least seven pages, we can construct a lattice stick version of $K$ using $(3)(3) +2 + 4 + 5 + 5 + 1 + 6(\alpha - 7) = 6\alpha - 16$ sticks. In the case that there are fewer than seven pages, the only knots generated are the trefoil knot and the figure-eight knot, both of which have lattice stick conformations that beat this bound.

\begin{figure}
 \centering
 \includegraphics[height=2.4in]{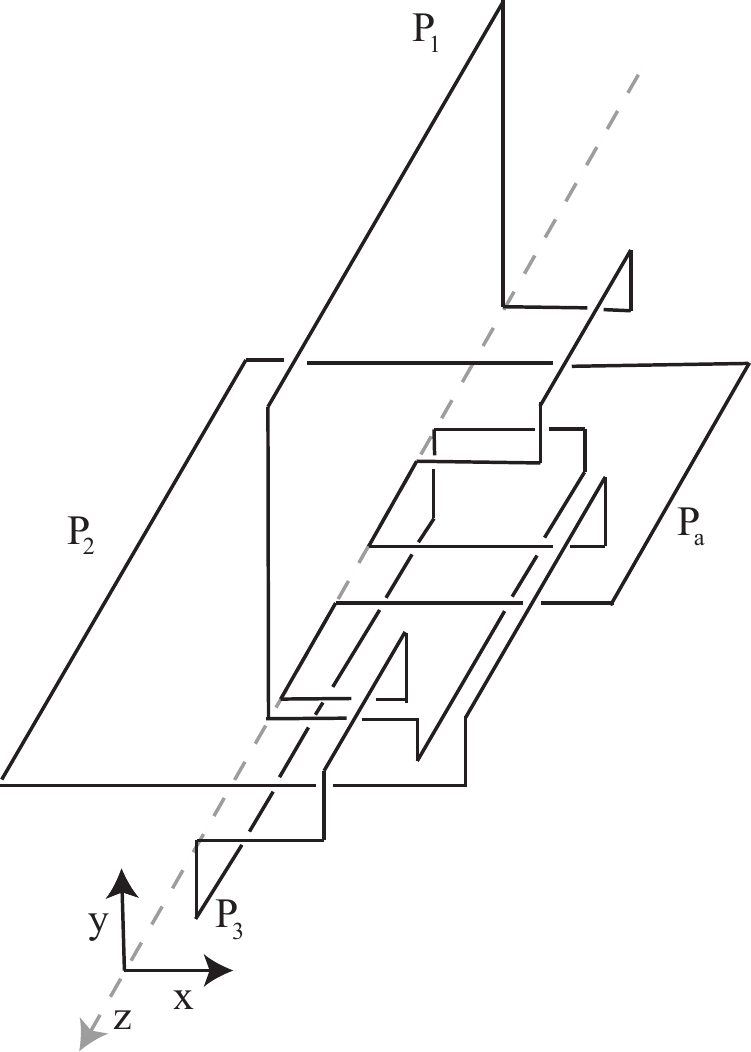}
 \caption{Constructing a cubic lattice conformation from an arc presentation.}
\label{Lattice Arc Index Diagram}
\end{figure}
\end{proof}

From \cite{Bae and Park}, we have $\alpha \leq c + 2$, which immediately yields the following corollary:

\begin{cor}
$s_{CL}[K] \leq 6c[K] -4$.
\end{cor}
 
\label{crossing}
\end{section}

\begin{section}{Links}

In this section, we find the lattice stick number for certain links of more than one component.

\begin{thm}
The lattice stick number of certain links are as listed in Table \ref{Flo:linktable}.

\begin{table}
\begin{centering}
\begin{tabular}{|c|c|}
\hline 
Links & $s_{CL}$\tabularnewline
\hline
\hline 
$0_{1}^{2}$ \Vsp & $8$\tabularnewline
\hline 
$2_{1}^{2}$ \Vsp & $8$\tabularnewline
\hline 
$7_{7}^{2}$ \Vsp & $16$\tabularnewline
\hline 
$8_{15}^{2}$ \Vsp & $18$\tabularnewline
\hline 
$8_{16}^{2}$ \Vsp & $18$\tabularnewline
\hline 
$6_{2}^{3}$ \Vsp & $12$\tabularnewline
\hline 
$6_{3}^{3}$ \Vsp & $12$\tabularnewline
\hline
\end{tabular}
\par\end{centering}
\caption{Lattice stick index of certain links.}
\label{Flo:linktable}
\end{table}
\end{thm}

\begin{proof}
Note that the bound $s_{CL}[K] \geq 6b[K]$ holds only for nontrivial knots, which cannot lie entirely in a plane. As a consequence, this bound does {\em not} necessarily hold for arbitrary links, since one or more of the components may be contained in planes.
However, if $K_1,\dots,K_n$ are the components of a link $L$, we {\em do} have $s_{CL}[L] \geq s_{CL}[K_1] + \dots + s_{CL}[K_n]$. This lower bound together with the lattice conformations presented in Figure \ref{Flo:links1}
complete the proof.

\begin{figure}
\includegraphics[height=3in]{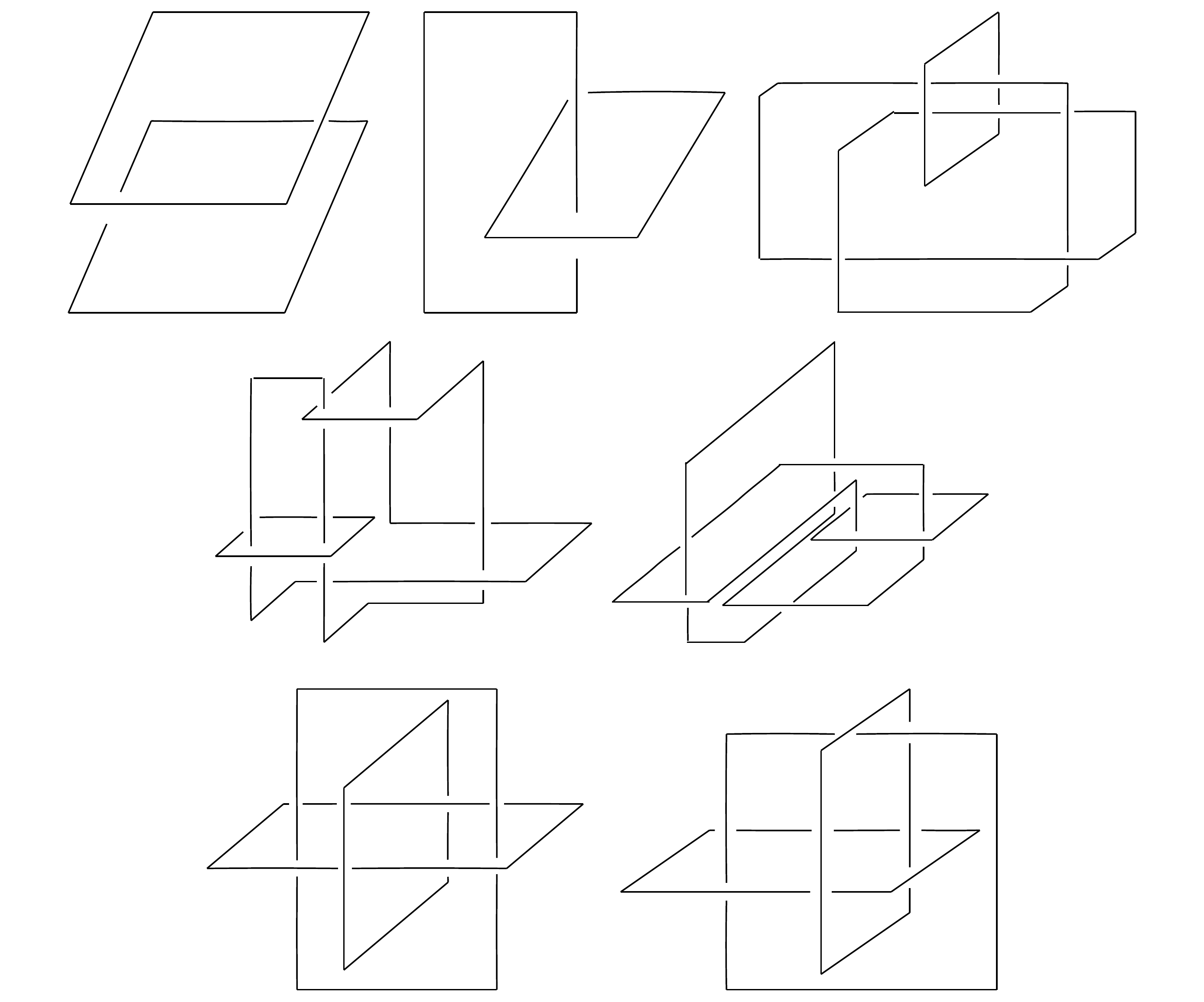}
\caption{Minimal lattice conformations of links $0_{1}^{2}$, $2_{1}^{2}$,
$7_{7}^{2}$, $8_{15}^{2}$, $8_{16}^{2}$, $6_{2}^{3}$, $6_{3}^{3}$.}
\label{Flo:links1}
\end{figure}
\end{proof}

\begin{proposition}
Let $L$ be a two-component link conformation with linking number $n$ and suppose that one of the components is entirely contained in a plane. Then the number of sticks in the conformation is at least $4n+4$.
\label{lem:Linkgeq4}
\end{proposition}

\begin{proof}
Consider the projection $P$ of $L$ down the $z$-axis, and let $P'$ be a slight perturbation of $P'$ as in Remark \ref{CanMakeKnotDiagram}. As $L$ has a linking number of $n$, so does $P'$. Let $K_1$ and $K_2$ denote the knot components of $L$ and assume $K_1$ lies in a plane. Giving $K_2$ an orientation, we can therefore assume that $K_2$ crosses as an understrand from the inside to the outside of $K_1$ at least $n$ times and crosses as an overstrand from the outside to the inside of $K_1$ at least $n$ times. Note that a single stick of $K_2$ makes at most a contributions of $1/2$ to the linking number of $P'$. Then $K_2$ requires at least $2n$ $xy$-sticks, each contributing $1/2$ to the linking number of $P'$, and a $z$-stick between consecutive $xy$-sticks in order to enable the crossings.
\end{proof}

\begin{cor}
Let $n$ be the linking number of a 2-braid link. Then the two-braid links, excluding $4_1^2$ and $6_1^2$, can be realized with $4n+4$ sticks. 
\end{cor}

\begin{proof} Note that the two exceptions have linking numbers  2 and 3, cases that are addressed in the next lemma. The Hopf link, with linking number 1,  can obviously be realized with 8 sticks, and Figure \ref{Flo:braids} demonstrates how one can construct the other 2-braid links except $4_1^2$ and $6_1^2$ with $4n+4$ sticks.
\end{proof}

\begin{figure}
\includegraphics[height=1.5in]{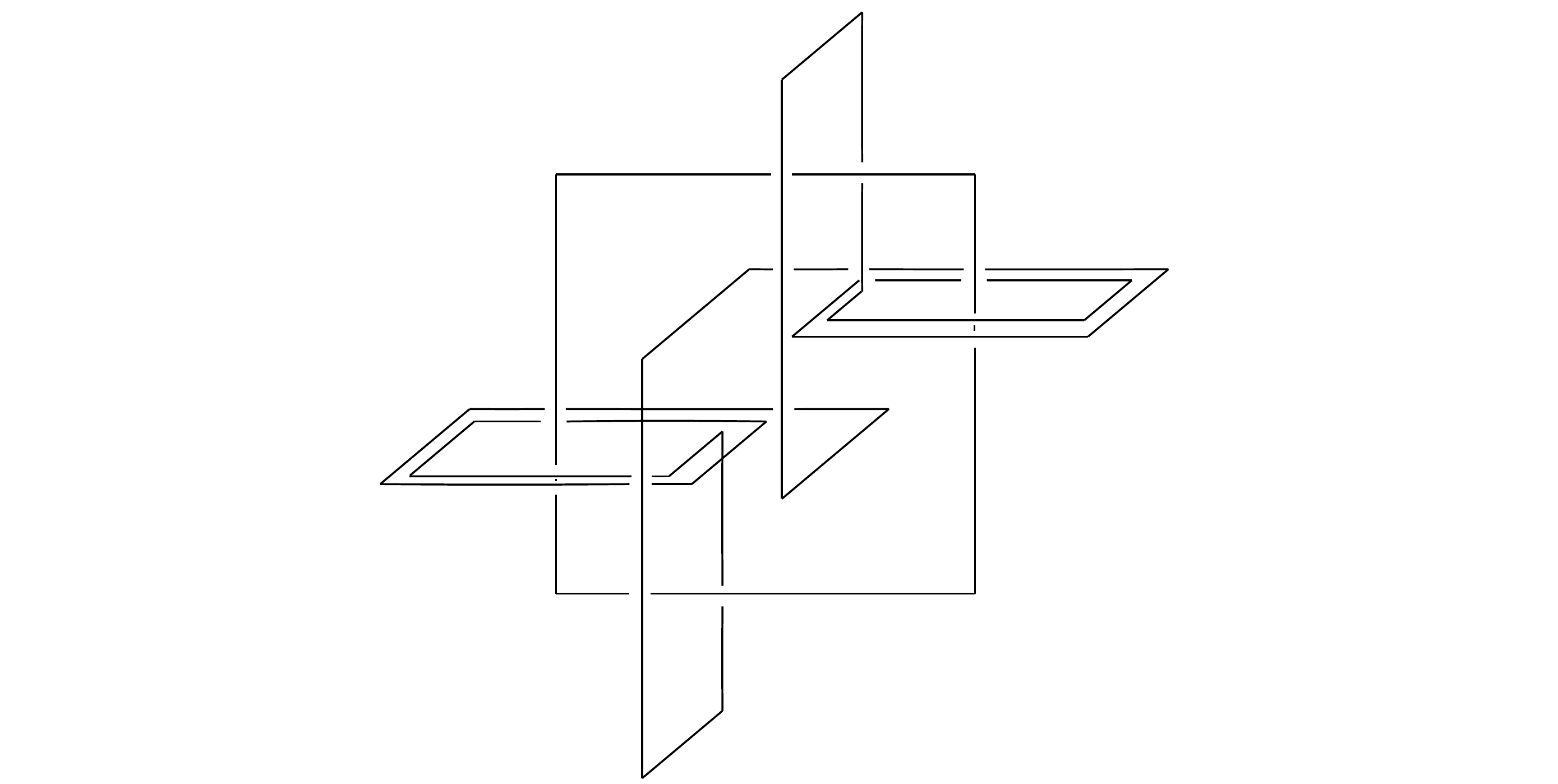}
\caption{Two-braid link with linking number $6$.}
\label{Flo:braids}
\end{figure}

\begin{lemma}
Let $L$ be a two-component cubic lattice link conformation with linking number $n=2$ or $3$ and suppose that one of the components lies entirely in a plane. Then the number of sticks in the conformation is at least $4n+5$.
\label{lem:LinkLowerBound}
\end{lemma}

\begin{proof}
Assume for contradiction that $L$ is composed of at most $4n + 4$ sticks, and as in the previous proof, let $K_{1}$ and $K_{2}$ denote the knot components of $L$, where $K_{1}$ lies entirely in the $xy$-plane and in the projection down the $z$-axis,  that $K_2$ crosses as an understrand from the inside to the outside of $K_1$ at least $n$ times and crosses as an overstrand from the outside to the inside of $K_1$ at least $n$ times. Then $K_2$ requires at least $2n$ $xy$-sticks each contributing $1/2$ to the linking number of $P'$, and we need a $z$-stick between these consecutive $xy$-sticks in order to enable the crossings, amounting to an additional $2n$ sticks. We can then assume that $K_1$ is a rectangle made from four sticks, and  the sticks of $K_2$ follow the repeating pattern: $xy$-stick passing over $K_1$ from outside to inside, descending $z$-stick, $xy$-stick passing
under $K_1$ from inside to outside, ascending $z$-stick, $xy$-stick passing over $K_1$ from outside to inside, etc.  A pair of  $xy$-sticks connected by a $z$-stick projecting to outside of $K_1$ must both be either $x$-sticks or $y$-sticks.

Since this accounts for all $4n+4$ sticks, the endpoints of all of the $2n$ $xy$-sticks that project to inside $K_1$ must all project to the same point, so as to avoid the necessity of additional $xy$-sticks to close up the knot. For convenience, assume that they project to the origin in the $xy$-plane. However then the $z$-sticks that connect these endpoints in pairs must all lie on the $z$-axis. But in order to obtain the requisite linking number, each interior endpoint of an understrand, which lies beneath the xy-plane, must connect by a $z$-stick to the interior endpoint of an overstrand, which lies above the $xy$-plane. However, at most one of the connecting $z$-sticks on the $z$-axis can pass through the origin, a contradiction.
\end{proof}

\begin{thm}
The lattice stick index of $4_{1}^{2}$ is $13$.
\end{thm}

\begin{proof}
The linking number of  $4_{1}^{2}$ is $n=2$. Suppose for contradiction that we have a cubic lattice conformation $L$ of $4_{1}^{2}$ with $12$ sticks or less. Let $K_{1}$ and $K_{2}$ denote the knot components of $L$. Consider the projection $P$ of $L$ into the $xy$-plane, and let $P'$ be a slight perturbation of $P$ as in Remark \ref{CanMakeKnotDiagram}.

\begin{figure}
\includegraphics[width=3.5in]{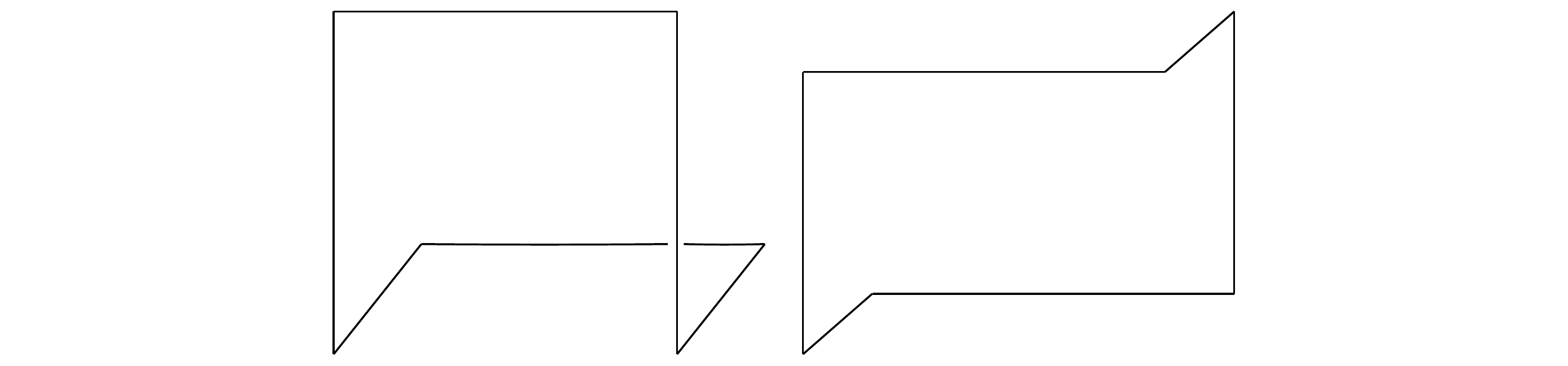}\caption{Six stick trivial components.}
\label{Flo:6sticks}
\end{figure}

By Lemma \ref{lem:LinkLowerBound}, neither component may lie entirely in a plane. Therefore both $K_1$ and $K_2$ are constructed from $6$ sticks. In particular, $K_1$ and $K_2$ are each constructed from $2$ parallel sticks in each direction, and it is not hard to see that the two constructions shown in Figure \ref{Flo:6sticks} are the only possible $6$-stick lattice conformations up to rotation and changing the lengths of sticks. We can assume that $K_1$ is positioned such that it $z$-projects to a rectangle. Then in the $xy$-plane it makes sense to talk about the inside and outside of $K_1$, and as before we can assume $K_2$ crosses as an understrand from the inside to the outside of $K_1$ at least twice and crosses as an overstrand from the outside to the inside of $K_1$ at least twice. A single stick of $K_2$ can cross $K_1$ in $P'$ at most twice, and such a stick may contribute $1$ to the linking number. If we have such a stick, say $x_{ij}$, then there must be a parallel stick which either makes the opposite contribution to the linking number or does not contribute to the linking number (since there are two sticks in each direction and $K_2$ is a closed loop). Either case is clearly impossible, since then the remaining four sticks of $K_2$ lie in either the plane $x=i$ or $x=j$ and therefore cannot contribute to the linking number. We conclude that each of the four $xy$-sticks of $K_2$ must contribute $1/2$ to the linking number. All four inner endpoints must project to the same $xy$ point in $P$, which is impossible without connecting oppositely oriented strands of $K_2$ or making $K_2$ into a two component link. 

Figure \ref{Fig:braids4,6} completes the proof.

\begin{figure}
\includegraphics[height=1.2in]{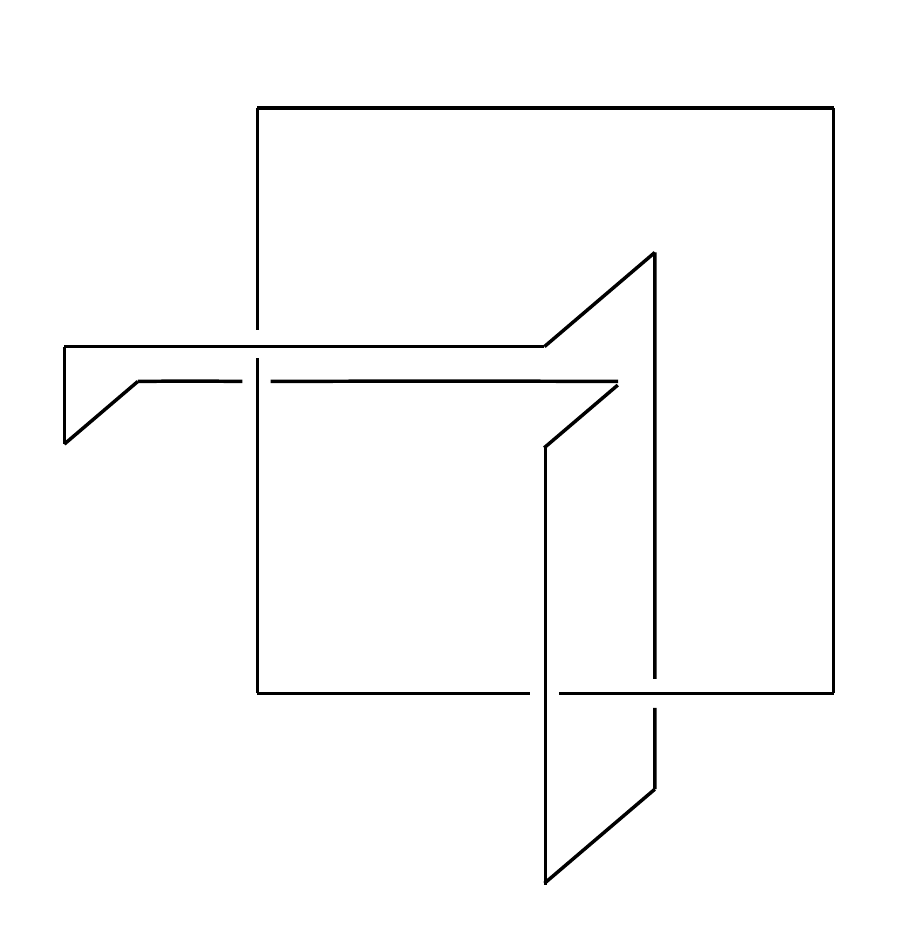}
\caption{Minimal lattice conformation of link $4_{1}^{2}$.}
\label{Fig:braids4,6}
\end{figure}
\end{proof}

\label{links}
\end{section}

\end{document}